\documentclass[12pt,a4paper,dvipdfm]{article}
\usepackage[margin=1 in]{geometry}
\usepackage{graphicx}
\usepackage{amssymb, amsmath}
\usepackage{amsthm}
\usepackage{dsfont}
\usepackage{enumerate}
\usepackage{setspace}
\begin{document}
\newcommand{\up}{\vspace*{-0.05cm}}
\allowdisplaybreaks[1]
\numberwithin{equation}{section}
\renewcommand{\theequation}{\thesection.\arabic{equation}}
\newtheorem{thm}{Theorem}[section]
\newtheorem{lemma}{Lemma}[section]
\newtheorem{pro}{Proposition}[section]
\newtheorem{prob}{Problem}[section]
\newtheorem{quest}{Question}[section]
\newtheorem{ex}{Example}[section]
\newtheorem{cor}{Corollary}[section]
\newtheorem{conj}{Conjecture}[section]
\newtheorem{cl}{Claim}[section]
\newtheorem{df}{Definition}[section]
\newtheorem{rem}{Remark}[section]
\newcommand{\beq}{\begin{equation}}
\newcommand{\eeq}{\end{equation}}
\newcommand{\<}[1]{\left\langle{#1}\right\rangle}
\newcommand{\be}{\beta}
\newcommand{\ee}{\end{enumerate}}
\newcommand{\Bul}{\mbox{$\bullet$ } }
\newcommand{\al}{\alpha}
\newcommand{\ep}{\epsilon}
\newcommand{\si}{\sigma}
\newcommand{\om}{\omega}
\newcommand{\la}{\lambda}
\newcommand{\La}{\Lambda}
\newcommand{\Ga}{\Gamma}
\newcommand{\ga}{\gamma}
\newcommand{\im}{\Rightarrow}
\newcommand{\2}{\vspace{.2cm}}
\newcommand{\es}{\emptyset}

\vspace{-2cm}
\markboth{R. Rajkumar and P. Devi}{Permutability graph of cyclic subgroups}
\title{\LARGE\bf Permutability graph of cyclic subgroups}
\author{R. Rajkumar\footnote{e-mail: {\tt rrajmaths@yahoo.co.in}},\ \ \
P. Devi\footnote{e-mail: {\tt pdevigri@gmail.com}}\\
{\footnotesize Department of Mathematics, The Gandhigram Rural Institute -- Deemed University,}\\ \footnotesize{Gandhigram -- 624 302, Tamil Nadu, India}\\[3mm]
}
\date{}
\maketitle
\begin{abstract}
Let $G$ be a group.
\textit{The permutability graph of cyclic subgroups of $G$}, denoted by $\Gamma_c(G)$, is a graph with all the
proper cyclic subgroups of $G$ as its vertices and two distinct vertices in $\Gamma_c(G)$ are adjacent if and only if
the corresponding subgroups permute in $G$.
In this paper, we classify the finite groups whose permutability graph of cyclic subgroups belongs to one of the following: bipartite, tree, star graph, triangle-free, complete bipartite, $P_n$, $C_n$, $K_4$,  $K_{1,3}$-free, unicyclic.
We classify abelian groups  whose permutability graph of cyclic subgroups are planar.
Also we investigate the connectedness, diameter, girth, totally disconnectedness, completeness and regularity of these graphs.
\end{abstract}

\section{Introduction} \label{sec:1}
The properties of a group can be studied by assigning a suitable graph to it and by analyzing the properties of the associated graphs
 using the tools of graph theory. The Cayley graph is a well known example of a graph associated to a group, which have been studied extensively in the
literature (see, for example, \cite{elena, li}).
In the past twenty five years many authors have assigned various graphs to  study some specific properties of groups . For instance, see~\cite{abdolla1, cameron, herz,  manz}.

Recall that two subgroups $H$ and $K$ of a group $G$ are said to \emph{permute} if $HK=KH$; equivalently $HK$ is a subgroup of $G$. In \cite{asc}, Aschbacher defined a graph corresponding to a group $G$ and for a fixed prime $p$, having all the subgroups of order $p$  as its vertices and two
vertices  are adjacent if they permute. To study the transitivity of permutability of subgroups,
Bianchi, Gillio and Verardi in \cite{binachi2},
defined a graph corresponding to a group $G$, called
the \emph{permutability graph of non-normal subgroups of $G$}, having all the proper non-normal subgroups of $G$ as its vertices and two
vertices  are adjacent if they permute (see, also in \cite{binachi3, gillio1}).  In \cite{raj},  the authors considered the generalized case of this graph, called
\emph{the permutability graph of subgroups of $G$}, denoted by $\Gamma(G)$, having the vertex set consisting of all proper subgroups
of $G$ and two vertices  are adjacent if they permute.

 In \cite[p.14]{bal}, Ballester-Bolinches et al introduced
a graph  corresponding to a group $G$, having all the cyclic subgroups of $G$ as it vertices and two
vertices  are adjacent if they permute. In this paper, as a particular case, we consider a graph, denoted by $\Gamma_c(G)$ with vertex set consists of all proper cyclic subgroups of $G$ and two vertices  are adjacent if they permute. We will call this graph as \emph{the permutability graph of cyclic subgroups of $G$}. By investigating the properties of this graph, we study the permutability of cyclic subgroups of the corresponding group. Especially, Theorems~\ref{ct14}, \ref{ct108} and \ref{ct51},  Corollaries~\ref{pc15} and  \ref{ct102} in this paper are some of the main applications for group theory.

Now we introduce some notion from graph theory that we will use in this
article. Let $G$ be a simple graph with vertex set $V(G)$ and edge set $E(G)$. $G$ is said to be \emph{complete} if any two of
its vertices are adjacent. A complete graph with $n$ vertices is denoted by $K_n$. $G$ is \textit{bipartite} if $V(G)$ is the union of two
disjoint sets $X$ and $Y$ such that no two vertices in the same subset are adjacent. Here $X$ and $Y$ are called a \textit{bipartition} of $G$.
A bipartite graph $G$  with bipartition $X$ and $Y$ is called
\textit{complete bipartite} if every vertex in $X$ is adjacent with every vertex in $Y$. If $|X|=m$ and $|Y|=n$, then the corresponding
graph is denoted by $K_{m,n}$. In particular, $K_{1,n}$ is called the \textit{star graph} and $K_{1,3}$ is called the \textit{claw graph}. A graph is \emph{planar} if it can be drawn in a plane so that no two edges intersect except possibly at vertices. The \textit{degree}
of the vertex $v$ in $G$ is the  number of edges incident with
$v$ and is denoted by $\deg_G(v)$. A graph is said to be
\textit{regular} if  degrees of
all the vertices are same. A \textit{path} joining two vertices $u$ and $v$ in $G$ is a finite sequence  $(u=) v_0, v_1, \ldots, v_n (=v)$ of
distinct vertices, except,
possibly, $u$ and $v$ such that $u_i$ is adjacent with $u_{i+1}$, for all $i=0, 1, \ldots , n-1$.
A path joining $u$ and $v$ is a cycle if $u=v$. The length of a path or cycle is the number of edges in it. A path or cycle of
length $n$ is denoted by $P_n$ or $C_n$ respectively. A graph with exactly one cycle is said to be \emph{unicyclic}.
A graph is a \emph{tree} if it has no cycles. The \emph{girth} of a graph $G$ is the length of the smallest cycle in it and is denoted by $girth(G)$.

A graph is said to be \textit{connected} if every pair of distinct vertices can be joined by a path. The \emph{distance} between two vertices
 $u$ and $v$ in $G$, denoted by $d(u,v)$, is the length of the shortest path between them, and $d(u,v)=0$ if $u=v$. If there exists no path between them,
 then we define $d(u,v)=\infty$. The \emph{diameter} of $G$, denoted by $diam(G)$ is the maximum distance between any two vertices in the graph.
 An \emph{isomorphism} of graphs $G_1$ and $G_2$ is an edge-preserving bijection
 between the vertex sets of $G_1$ and $G_2$.  $G$ is said to be
$H$-\textit{free} if $G$ has no induced subgraph isomorphic to $H$. Let $G_1 =(V_1, E_1)$ and $G_2 = (V_2, E_2)$ be two
simple graphs. Their \textit{union} $G_1 \cup G_2$ is a graph with
vertex set $V_1 \cup V_2$ and edge set $E_1 \cup E_2$. Their \textit{join} $G_1 + G_2$ is a graph consist of $G_1 \cup G_2$ together with all the lines
joining points of $V_1$ to points of $V_2$. For any connected graph $G$, we write $nG$ for the graph with $n$ components each isomorphic to $G$.
 For basic graph theory terminology, we refer to \cite{harary}.

The dihedral group of order $2n$, $n\geq 3$ is defined by $D_{2n}= \langle a, b~|~ a^{n}= b^2= 1, ab= ba^{-1} \rangle$. For any integer $n \geq 2$,
the generalized Quaternion group of order $4n$ is given by
$Q_{4n} = \big < a, b ~|~a^{2n} = b^4 = 1, a^n = b^2 = 1, bab^{-1} = a^{-1}\big >$.
The modular group of order $p^\alpha$, $\alpha \geq 3$ is given by
$M_{p^\alpha}=\langle a,b~|~a^{p^{\alpha-1}}=b^p=1, bab^{-1}=a^{p^{\alpha-2}+1}\rangle$.
 For an integer $n \geq 1$,
$S_n$ and $A_n$ denotes the symmetric group and alternating group of degree $n$ acting on
$\{1,2,\ldots,n\}$ respectively. If $n$ is a any positive integer, then $\tau(n)$ denotes the number of positive divisors of $n$.
We denote the order of an element $a \in \mathbb Z_n$ by $\text{ord}_n(a)$. The number of Sylow $p$-subgroups of a group $G$
is denoted by $n_p(G)$; or simply by $n_p$ if there is no ambiguity.

The rest of the paper is arranged as follows: In Section~\ref{sec:2}, we study some basic properties of permutability graph
of cyclic subgroups of groups.

Section~\ref{mainsec:3} gives the classification of finite groups
whose permutability graphs of cyclic subgroups are one of the following: bipartite, tree, star graph, triangle-free, complete bipartite,
 $P_n$, $C_n$, $K_4$,  $K_{1,3}$-free, unicyclic. We estimate the girth of the permutability graphs of cyclic subgroups of finite groups. We also characterize the groups having totally disconnected permutability graphs of cyclic subgroups.

 In Section~\ref{sec:9}, we investigate connectedness, diameter, regularity, completeness of the permutability graph of cyclic subgroups
of a given group.  Also we classify  abelian groups whose permutability graph of cyclic subgroups are planar. We characterize the groups $Q_8$, $S_3$ and $A_4$
by using their permutability graph of cyclic subgroups. Moreover, we pose some open problems in this section.


We recall the following theorem, which we will  use in the subsequent sections.
%
%
\begin{thm}\label{pct100}\textbf{(\cite[Corollary 5.1]{raj})}
Let $G$ be a finite group and $p$, $q$ be distinct primes. Then
\begin{itemize}
\item [(i)] $\Gamma(G)$ is $C_n$ if and only if $n=3$ and $G$ is either $\mathbb Z_{p^4}$ or $\mathbb Z_2\times \mathbb Z_2$;

\item [(ii)] $\Gamma(G)$ is $P_n$ if and only if  $n=1$ and $G$ is either $\mathbb Z_{p^3}$ or $\mathbb Z_{pq}$;

\item [(iii)] $\Gamma(G)$ is claw-free if and only if $G$ is either $\mathbb Z_{p^\alpha}$ $(\alpha=2,3,4)$ or $\mathbb Z_{pq}$.
\end{itemize}
\end{thm}

\section{Some basic results} \label{sec:2}

Note that the only groups having  no  proper cyclic subgroups are the trivial group, and the groups of prime order, so
 it follows that, we can define $\Gamma_c(G)$ only when the group $G$ is not isomorphic to either of these groups.

In this section, we study some basic properties about of permutability graph of cyclic subgroups of a given group.
We start with the following result whose proof is immediate.

\begin{lemma}\label{permutability cyclic l2}
Let $G$ be a group. If $G$ has $r$ proper cyclic subgroups, which are permutes with each other, then $\Gamma_c(G)$ has $K_r$ as a subgraph.
\end{lemma}

\begin{thm}\label{ct20}
 Let $G_1$ and $G_2$ be two groups. If $G_1\cong G_2$, then $\Gamma_c(G_1)\cong \Gamma_c(G_2)$.
\end{thm}
\begin{proof}
  Let $f:G_1\rightarrow G_2$ be a group isomorphism. Define a map $\psi:V(\Gamma_c(G_1))\rightarrow V(\Gamma_c(G_2))$ by $\psi(H)=f(H)$,
for every $H\in V(\Gamma_c(G_1))$.
Then it is easy to see that $\psi$ is  a graph isomorphism.
\end{proof}

\begin{rem}The converse of Theorem~\ref{ct20} is not true. For example, consider the non-isomorphic groups $G_1=\mathbb Z_{p^5}$, where $p$ is a prime and
$G_2=\mathbb Z_3\times \mathbb Z_3$. Here $G_1$ has subgroups $\mathbb Z_{p^i}$, $i=1,2,3,4$ and $G_2$ has proper cyclic subgroups
$\langle (1,0)\rangle$, $\langle (x,1)\rangle$, $x=0,1,2$.
It follows that $\Gamma_c(G_1)\cong K_4\cong \Gamma_c(G_2)$.
\end{rem}

\begin{thm}\label{ct1}
If $G$ is a group and $N$ is a subgroup of $G$, then $\Gamma_c(N)$ is a subgraph of $\Gamma_c(G)$.
\end{thm}

\section{Some classification related results for $\Gamma_c(G)$}\label{mainsec:3}
The aim of this section is to classify the solvable  groups
whose permutability graphs of cyclic subgroups are one of the following: bipartite, complete bipartite, tree, star graph, $C_3$-free,
$C_n$, $K_4$, $P_n$, $K_{1,3}$-free, unicyclic. First we consider the finite groups and then we deal with the infinite groups.

\subsection{Finite abelian groups}\label{sec:3}

\begin{pro}\label{ct2}
 Let $G$ be a finite abelian group and $p$, $q$ be distinct primes. Then
\begin{enumerate}[{\normalfont (i)}]
\item $\Gamma_c(G)$ is $C_3$-free if and only if $G$ is either $\mathbb Z_{p^{\alpha}}$ $(\alpha = 2,3)$ or $\mathbb Z_{pq}$;
\item $\Gamma_c(G)$ is bipartite if and only if it is $C_3$-free;
\item $\Gamma_c(G)$ is $C_n$ if and only if $n=3$ and $G$ is either $\mathbb Z_{p^4}$ or $\mathbb Z_2\times \mathbb Z_2$;
\item $\Gamma_c(G)$ is $P_n$ if and only if $n=1$ and $G$ is either $\mathbb Z_{p^3}$ or $\mathbb Z_{pq}$;
\item $\Gamma_c(G)$ is $K_4$ if and only if $G$ is one of $\mathbb Z_{p^5}$, $\mathbb Z_{p^2q}$, $\mathbb Z_3\times \mathbb Z_3$;
\item $\Gamma_c(G)$ is claw-free if and only if $G$ is one of $\mathbb Z_{p^\alpha}$ $(\alpha=2,3,4)$, $\mathbb Z_{pq}$,
$\mathbb Z_2\times \mathbb Z_2$;
\item $\Gamma_c(G)$ is unicyclic if and only if $G$ is either $\mathbb Z_{p^4}$ or $\mathbb Z_2\times \mathbb Z_2$.
\end{enumerate}
\end{pro}
\begin{proof} Let $|G|=p_1^{\alpha_1}p_2^{\alpha_2}\ldots p_k^{\alpha_k}$, where $p_i$'s are distinct primes
and $\alpha_i\geq 1$ for every $i=1,2, \ldots k$.
We divide the proof into two cases.

\noindent \textbf{Case 1:} If $G$ is cyclic, then $\Gamma_c(G)\cong \Gamma(G)$.
So in view of this fact and by the proof of \cite[Theorem 3.1]{raj}, we have
\begin{equation}\label{e5}
 \Gamma_c(G)\cong K_r,
\end{equation}
 where $r$ is the number of proper subgroups of $G$, which is given by $r= (\alpha_1 + 1)(\alpha_2 + 1) \cdots (\alpha_k +1)- 2$.
It follows that $\Gamma_c(G)\cong K_4$ if and only if $G$ is one of $\mathbb Z_{p^5}$ or $\mathbb Z_{p^2q}$. Furthermore, $\Gamma_c(G)$ is bipartite or $C_3$-free if and only if $G$ is either $\mathbb Z_{p^{\alpha}}$ $(\alpha = 2,3)$ or $\mathbb Z_{pq}$.
Note that the bipartiteness and $C_3$-freeness of permutability graphs of finite cyclic groups were proved  in \cite[Proposition 3.1 and corollary 3.1]{raj1}. We repeated them here for the sake of completeness.
Also by Theorem~\ref{pct100}, we have
 \begin {itemize}
\item [(i)] $\Gamma_c(G)$ is $C_n$ if and only if $n=3$ and $G\cong\mathbb Z_{p^4}$.

\item [(ii)] $\Gamma_c(G)$ is $P_n$ if and only if  $n=1$ and $G$ is either $\mathbb Z_{p^3}$ or $\mathbb Z_{pq}$.

\item [(iii)] $\Gamma_c(G)$ is claw-free if and only if $G$ is one of $\mathbb Z_{p^\alpha}$ $(\alpha=2,3,4)$, $\mathbb Z_{pq}$.
\end{itemize}

\noindent \textbf{Case 2:} If $G$ is non-cyclic, then we have the following cases to consider:

\noindent \textbf{Subcase 2a:} $k=1$. If $\alpha_1>2$, then $G$ has a subgroup isomorphic to either $\mathbb Z_{p}\times \mathbb Z_{p}\times \mathbb Z_{p}$ or
$\mathbb Z_{p^2}\times \mathbb Z_{p}$, for some prime $p$. It is easy to see that these groups have at least five proper cyclic subgroups, so they form
$K_5$ as a subgraph of $\Gamma_c(G)$.
If $\alpha_1=2$, then $G\cong \mathbb Z_p\times \mathbb Z_p$, for some prime $p$. But the number of nontrivial subgroups of
$\mathbb Z_p\times \mathbb Z_p$ is $p+1$; they are $\langle(1,0)\rangle$, $\langle (a,1)\rangle$, for each $a\in \{0,1,2,\ldots,p-1\}$. Thus, by
Lemma~\ref{permutability cyclic l2},
\begin{equation}\label{e1}
\Gamma_c(G)\cong K_{p+1}.
\end{equation}
 Therefore, $\Gamma(G_1)$ contains $C_3$ as a subgraph; it is $C_3$ if and only if $p=2$; it is
$K_4$ if and only if $p=3$; it is claw-free if and only if $p=2$.

\noindent \textbf{Subcase 2b:} $k>1$. If $\alpha_i>1$ for some $i$, then $G$ has a subgroup $H$ isomorphic to $\mathbb Z_{pq}\times \mathbb Z_{p}$,
for some distinct primes $p$ and $q$. It is easy to
see that $H$ has at least five proper cyclic subgroups, so they form $K_5$ as a subgraph of $\Gamma_c(G)$.

The proof follows by combining these cases.
\end{proof}

\subsection{Finite non-abelian groups}\label{sec:4}

\begin{pro}\label{ct5}
 Let $G$ be a non-abelian of order $p^{\alpha}$, where $p$ is a prime and $\alpha\geq3$. Then $\Gamma_c(G)$ contains $C_3$ and $K_{1,3}$ as proper
subgraphs; $\Gamma_c(G)\cong K_4$ if and only if $G\cong Q_8$.
\end{pro}
\begin{proof}
We first prove this result when $\alpha=3$.
According to the Burnside \cite{burn}, up to isomorphism there are only four non-abelian groups of order $p^3$, where $p$ is a prime,
namely $Q_8$, $M_8$, $M_{p^\alpha}$ and
$(\mathbb Z_p\times \mathbb Z_p)\rtimes \mathbb Z_p$, $p >2$. If $G\cong Q_8$, then
by \cite[Theorem 4.3~]{raj}, we have
\begin{equation}\label{e2}
 \Gamma_c(G)\cong K_4.
\end{equation}
 If $G\cong M_8$, then
$H_1:=\langle a\rangle$, $H_2:=\langle a^2\rangle$, $H_3:=\langle b\rangle$, $H_4:=\langle ab\rangle$, $H_5:=\langle a^2b\rangle$ are proper cyclic
subgroups of $G$, so $|V(\Gamma_c(G))|\geq 5$.
Since $H_1$, $H_2$ are normal in $G$, they permutes with all the subgroups of $G$. Thus, $\Gamma_c(G)$ has $C_3$ as a subgraph induced by the
vertices $H_1$, $H_2$, $H_3$; but it is not $K_4$ as it has five vertices.
Also $K_{1,3}$ is a subgraph of $\Gamma_c(G)$ with bipartition $X:=\{H_1\}$ and $Y:=\{H_2$, $H_3$, $H_4\}$.
If $G\cong M_{p^\alpha}$, where $p$ is a prime and $p>2$, then $H_1:=\langle a\rangle$, $H_2:=\langle ab\rangle$, $H_3:=\langle ab^2\rangle$,
$H_4:=\langle b\rangle$,
$H_5:=\langle a^p\rangle$ are proper cyclic subgroups of $G$, so $|V(\Gamma_c(G))|\geq 5$.
Here any two subgroups of $G$ permutes, so $K_5$ is a subgraph of $\Gamma_c(G)$.
If $G\cong (\mathbb Z_p\times \mathbb Z_p)\rtimes \mathbb Z_p$, then $\mathbb Z_p\times \mathbb Z_p$ is a subgroup of $G$ and since $p>2$, so by \eqref{e1},
 $\Gamma_c(G)$ contains $K_4$ as a proper subgraph. Clearly $|V(\Gamma_c(G))|\geq 5$.

Now we prove this result when $\alpha\geq 4$. We need to consider the following two cases:

\noindent\textbf{Case 1:} $G\cong Q_{2^\alpha}$. Then $G$ has two subgroups each isomorphic to $Q_8$, so in the view of \eqref{e2}, $\Gamma_c(G)$ contains $C_3$ and $K_{1,3}$
as proper subgraphs. Also $G$ has at least five proper cyclic subgroups, so $|V(\Gamma_c(G))|\geq 5$.

\noindent\textbf{Case 2:} $G\ncong Q_{2^\alpha}$. By \cite[Proposition 1.3]{scott}, the number of subgroups of order $p$ of $G$ is not unique and so by
\cite[Theorem IV, p.129]{burn}, $G$ has at least three subgroups, say $H_i$, $i=1,2,3$ of order $p$; also it has a subgroup,
say $H$ of order $p^3$. Suppose $\Gamma_c(H)$ contains $C_3$ and $K_{1,3}$; also $|V(\Gamma_c(H))|\geq 5$, then
$\Gamma_c(G)$ also has the same. So by Propositions~\ref{ct2} and \ref{ct5}, the only cases remains to check are
$H\cong\mathbb Z_{p^3}$ or $Q_8$. If $H\cong \mathbb Z_{p^3}$, then by~\eqref{e1}, $\Gamma_c(H)\cong K_2$, so $H$ together with its subgroups forms $C_3$ as
a subgraph of $\Gamma_c(G)$. The cyclic subgroups of $H$ together with the subgroups $H_i$'s make $|V(\Gamma_c(G))|\geq 5$. By
\cite[Corollary of Theorem IV, p.129]{burn}, $G$ has a normal subgroup of order $p$, without loss of generality, say $H_1$. Then  $K_{1,3}$ is a subgraph of $\Gamma_c(G)$ with
bipartition $X:=\{H_1\}$ and $Y:=\{H, H_2, H_3\}$. If $H\cong Q_8$, then by \eqref{e2}, $\Gamma_c(H)\cong K_4$. Also the cyclic subgroups of $H$
together with $H_i$'s also make $|V(\Gamma_c(G))|\geq 5$.

The proof follows by combining all the above.
\end{proof}

\begin{pro}\label{ct6}
 Let $G$ be the non-abelian group of order $pq$, where $p,q$ are distinct primes and $p<q$. Then
 $\Gamma_c(G)\cong K_{1,q}$.
\end{pro}
\begin{proof}
We have $G \cong \mathbb Z_q \rtimes \mathbb Z_p$.
Here every subgroup of $G$ is cyclic, so $\Gamma_c(G)\cong \Gamma(G)$.
By the proof of Theorem~4.4 in \cite{raj}, we have
 \begin{equation}\label{e3}
  \Gamma_c(G) \cong K_{1,q}.
 \end{equation}
 This completes the proof.
\end{proof}

Consider the semi-direct product $\mathbb Z_q \rtimes_{t} \mathbb Z_{p^{\alpha}} = \langle a,b | a^q= b^{p^{\alpha}}= 1, bab^{-1}= a^i,
{ord_{q}}(i)= p^t \rangle$, where $p$ and $q$ are distinct primes with $p^t~|~(q-1)$, $t \geq 0$. Then every semi-direct product $Z_q \rtimes Z_{p^{\alpha}}$
 is  one of these types \cite[Lemma 2.12]{boh-reid}. In the
future, when $t = 1$ we will suppress the subscript.

\begin{pro}\label{ct7}
Let $G$ be  a non-abelian group of order $p^2q$, where $p,q$ are distinct primes. Then
$\Gamma_c(G)$ contains $C_3$ as a proper subgraph; it is $K_{1,3}$-free if and only if $G\cong A_4$; it has at least five vertices.
\end{pro}
\begin{proof}
  Here we use the classification of groups of order $p^2q$ given in \cite[p.~76-80]{burn}.
We have the following cases to consider:

\noindent \textbf{Case 1:} $ p< q$:

\noindent \textbf{Case 1a:}  $p \nmid (q-1)$. By Sylow's Theorem, it is easy to see that there is no non-abelian group in this case.

\noindent \textbf{Case 1b:}  $p ~|~ (q-1)$, but $p^2 \nmid (q-1)$. In this case, there are two non-abelian groups.

The first group is $G_1:=\mathbb Z_q\rtimes\mathbb Z_{p^2}=\langle a,b~|~a^q=b^{p^2}=1,{bab}^{-1}=a^i,ord_q(i)=p\rangle$.
 It has $H_1:=\langle a\rangle$, $H_2:=\langle ab^p\rangle$, $H_3:=\langle b\rangle$, $H_4:=\langle b^p\rangle$, $H_5:=\langle ab\rangle$
as its proper cyclic subgroups, so $|V(\Gamma_c(G_1))|\geq 5$. Here $H_1$ and $H_2$ are normal in $G$, so they permutes with all the subgroups of $G$;
$H_4$ is a subgroup
of $H_3$ and $H_5$. So $K_4$ is a subgraph of $\Gamma_c(G_1)$ induced by $H_i$, $i=1,2,3,4$.

The second group in this case is $G_2:=\langle a,b,c~|~a^q=b^p=c^p=1,bab^{-1}=a^i,ca=ac,cb=bc,{ord_q}(i)=p\rangle$.
It has $H_1:=\langle a\rangle$, $H_2:=\langle b\rangle$, $H_3:=\langle c\rangle$, $H_4:=\langle bc\rangle$,
$H_5:=\langle ab\rangle$ as its proper cyclic subgroups, so $|V(\Gamma_c(G_2))|\geq 5$. Here $H_3$ permutes with all the subgroups of $G_2$; $H_2$, $H_3$, $H_4$
permutes with each other. So $C_3$ is a subgraph of $\Gamma_c(G_2)$ induced by the vertices $H_2$, $H_3$, $H_4$; and $K_{1,3}$ is a subgraph of
$\Gamma_c(G_2)$ with bipartition $X:=\{H_3\}$ and $Y:=\{H_1, H_2, H_3\}$.

\noindent \textbf{Case 1c:} $p^2~|~(q-1)$.  In this case, we have both groups $G_1$ and  $G_2$ from Case 1b
 together with the group $G_3:=\mathbb Z_q{\rtimes_{2}}\mathbb Z_p=\langle a,b~|~a^q=b^{p^2}=1,bab^{-1}=a^i,{ord_{q}}(i)=p^2\rangle$.
But in Case 1b, we already dealt with $G_1$ and $G_2$. Now we consider $G_3$. It has $H_1:=\langle a\rangle$, $H_2:=\langle b\rangle$,
 $H_3:=\langle b^p\rangle$, $H_4:=\langle ab\rangle$, $H_5:=\langle a^2b\rangle$ as its proper cyclic subgroups, so $|V(\Gamma_c(G_3))|\geq 5$.
Since $H_1$ is normal in $G_3$,
it permutes with all the subgroups of $G_3$; $H_3$ is a subgroup of $H_2$. So $C_3$ is a subgraph of $\Gamma_c(G)$ induced by $H_1$, $H_2$, $H_3$ and
$K_{1,3}$ is a subgraph of $\Gamma_c(G_3)$ with bipartition $X:=\{H_1\}$ and $Y:=\{H_2$, $H_3$, $H_4\}$.

\noindent \textbf{Case 2:} $ p> q$:

\noindent \textbf{Case 2a:} $ q ~\nmid~ (p^{2}-1)$. In this case there is no non-abelian group.

\noindent \textbf{Case 2b:} $ q~|~(p-1)$. In this case there are two groups. The first one is
$G_4:=\langle a,b~|~ a^{p^2}=b^q=1,bab^{-1}= a^i,{ord_{p^2}}(i)=q\rangle$. It has $H_1:=\langle a\rangle$, $H_2:=\langle a^p\rangle$,
 $H_3:=\langle a^pb\rangle$, $H_4:=\langle b\rangle$, $H_5:=\langle ab\rangle$ as its proper cyclic subgroups, so $|V(\Gamma_c(G_4))|\geq 5$.
Since $H_1$ is a normal subgroup of $G_4$,
so it permutes with all the subgroup of $G_4$; $H_2H_3=\langle a^p,b\rangle=H_2H_4$; $H_2H_5=\langle a^p,ab\rangle$. So $C_3$ is a subgraph of
$\Gamma_c(G_4)$ induced by $H_1$, $H_2$, $H_4$; $K_{1,3}$ is a subgraph of $\Gamma_c(G_4)$
with bipartition $X:=\{H_1\}$ and $Y:=\{H_2$, $H_3$, $H_4\}$.

Next, we have the family of groups $\langle a,b,c~|~a^p=b^p=c^q=1,cac^{-1}=a^i,cbc^{-1}=b^{i^t},ab=ba,{ord_{p}}(i)=q\rangle$.
 There are $(q+3)/2$ isomorphism types in this family (one for $t=0$ and one for each pair $\{x$, $x^{-1}\}$ in $\mathbb F^{\times}_p$.
We will refer to all of these groups as $G_{5(t)}$ of order $p^2q$.
They have a subgroup $H$ isomorphic to $\mathbb Z_p\times \mathbb Z_p$. Since $p>2$, so by \eqref{e1}, $\Gamma_c(G_{5(t)})$ contains $K_4$ as
a subgraph. In addition to these four vertices, $\Gamma_c(G_{5(t)})$ have $\langle c\rangle$ as their vertex, so $|V(\Gamma_c(G_{5(t)}))|\geq 5$.

\noindent \textbf{Case 2c:}  $q ~|~ (p+1)$. In this case, we have only one group of order
${p^2}q$, given by $G_6 := (\mathbb Z_p \times \mathbb Z_p) \rtimes \mathbb Z_q = \langle a,b,c~|~a^p= b^p= c^q= 1, ab= ba, cac^{-1}= a^{i}b^{j},
cbc^{-1}= a^{k}b^{l} \rangle$, where $\bigl(\begin{smallmatrix}
  i & j\\ k & l
\end{smallmatrix} \bigr)$ has order $q$ in $GL_2(p)$.  It has a subgroup $H$ isomorphic to $\mathbb Z_p \times \mathbb Z_p$. Since $p>2$, so by \eqref{e1},
$\Gamma_c(G_6)$ contains $K_4$ as a subgraph. In addition to these four vertices, $\Gamma_c(G_6)$ has $\langle c\rangle$ as its vertex,
so $|V(\Gamma_c(G_6))|\geq 5$.

Note that if $(p,q)=(2,3)$, the Cases 1 and 2 are not mutually exclusive. Up to isomorphism, there
are three non-abelian groups of order 12: $\mathbb Z_3\rtimes\mathbb Z_4$, $D_{12}$, and $A_4$.
In Case 1b we already dealt with $\mathbb Z_3 \rtimes \mathbb Z_4$ (the group $G_1$), and  $D_{12}$ (the group $G_2$).
 But for the case of $A_4 $ (the group $G_6$), we can not use the
argument as in Case 2c, since $p=2$. So  we now  separately deal with this case. Note that $A_4 \cong \mathbb (Z_2 \times \mathbb Z_2) \rtimes \mathbb Z_3$.
Here $H_1:=\mathbb Z_2 \times \mathbb Z_2$ is a subgroup of $A_4$ of order 4, and it has three nontrivial subgroups, say
$H_i$, $i=2,3,4$ each of order 2. Also $A_4$ has four subgroups of order 3, let them be $H_j$, $j=5,6,7,8$.
These eight subgroups are the only proper subgroups of $A_4$, so $|V(\Gamma_c(G))|\geq 5$.
Further,  $H_2$, $H_3$ and  $H_4$  permutes with each other, but no two subgroups $H_5$, $H_6$, $H_7$, $H_8$ permutes; for if they permutes, then $G$
has a subgroup of order 9, which is not possible.
Also, no $H_i$ $(i = 2,3, 4)$ permutes with $H_j$ $(j=5,6,7,8)$; for if they permutes, then $G$ has a subgroup of order 6, which is not
possible. Thus,
\begin{equation}\label{e4}
 \Gamma_c(G_6)\cong K_3\cup\overline{K}_4.
\end{equation}
The proof follows by combining all the cases.
\end{proof}

\begin{pro}\label{ct8}
 If $G$ is a non-abelian group of order $p^\alpha q$, where $p$, $q$ are two distinct primes with $\alpha\geq 3$, then
$\Gamma_c(G)$ has $C_3$ and $K_{1,3}$ as proper subgraphs; it has at least five vertices.
\end{pro}
\begin{proof}
 Let $P$ denote a Sylow $p$-subgroup of $G$. We first prove this result
for $\alpha=3$. If $p>q$, then $n_p=1$, by Sylow's Theorem and our
group $G\cong P\rtimes \mathbb Z_q$. Suppose $\Gamma_c(P)$ contains $C_3$ and $K_{1,3}$; $|V(\Gamma_c(G))|\geq 5$, then
$\Gamma_c(G)$ also has the same. So by Propositions~\ref{ct2} and \ref{ct5}, the only possibilities are
$P\cong\mathbb Z_{p^3}$ or $Q_8$. If $P\cong \mathbb Z_{p^3}$, then $G\cong \mathbb Z_{p^3}\rtimes \mathbb Z_q=\langle a,b~|~a^{p^3}=q=1, bab^{-1}=a^i,
ord_{p^3}(i)=q\rangle$ and it has $H_1:=\langle a\rangle$, $H_2:=\langle
 a^p\rangle$ ,$H_3:=\langle a^{p^2}\rangle$, $H_4:=\langle b\rangle$, $H_5:=\langle ab\rangle$ as its proper cyclic subgroups, so $|V(\Gamma_c(G))|\geq 5$.
Here $H_1$, $H_2$, $H_3$
are normal in $G$, so they permutes with all the subgroups of $G$. It follows that $\Gamma_c(G)$ contains $K_4$ as a proper subgraph.
If $P\cong Q_8$, then by~\eqref{e2}, $\Gamma_c(P)
\cong K_4$. But this $K_4$ is a proper subgraph of $\Gamma_c(G)$, since $G$ has a cyclic subgroup isomorphic to $\mathbb Z_q$, in addition and
so $|V(\Gamma_c(G))|\geq 5$.

Now, let us  consider the case $p< q$ and $(p,q)\neq (2, 3)$. Here $n_{q} = p$ is not possible. If $n_{q}= p^{2}$, then
 $q~|~(p+1)(p-1)$ which implies that $q| (p+1)$ or $q~|~(p-1)$. But this is impossible, since $q>p>2$. If $n_{q}= p^3$,
 then there are $p^{3}(q-1)$ elements of order $q$. But this only leaves $p^{3}q-p^{3}(q-1)= p^3$ elements, and the
Sylow $p$-subgroup must be normal, a case we already considered. Therefore, the only remaining possibility is that
$G\cong \mathbb Z_q \rtimes P$. Suppose $\Gamma_c(P)$ contains $C_3$ and $K_{1,3}$; $|V(\Gamma_c(P))|\geq 5$, then
$\Gamma_c(G)$ also has the same. So by Propositions~\ref{ct2} and \ref{ct5}, we have the only possibilities $P\cong\mathbb Z_{p^3}$ or
$Q_8$. If $P\cong \mathbb Z_{p^3}$, then $G\cong \mathbb Z_q\rtimes \mathbb Z_{p^3}=\langle a,b~|~a^q=b^{p^3}=1, bab^{-1}=a^i, ord_{q}(i)=p^3\rangle$
and it has $H_1:=\langle a\rangle$, $H_2:=\langle b\rangle$, $H_3:=\langle b^p\rangle$,
 $H_4:=\langle b^{p^2}\rangle$, $H_5:=\langle ab^p\rangle$ as its proper cyclic subgroups, so $|V(\Gamma_c(G))|\geq 5$.
Here $H_1$, $H_5$ are normal in $G$, so they
permutes with all the subgroups of $G$; $H_3$ is a subgroups of $H_2$. So $K_4$ is a subgraph of $\Gamma_c(G)$ induced by
$H_1$, $H_2$, $H_3$, $H_4$. The case $P\cong Q_8$ is similar to the earlier case.

If $(p,q)=(2,3)$, then $G\cong S_4$ and it has a subgroup $H$ isomorphic to $D_8$. Therefore, by
Theorem~\ref{ct5},
$\Gamma_c(H)$ contains $C_3$ and $K_{1,3}$ as proper subgraphs. Also $H$ has more than four cyclic subgroups, so $\Gamma_c(G)$ also has the same properties.

If $\alpha\geq4$, then $G$ has a subgroup, say $H$ of order $p^4$.
Suppose $\Gamma_c(H)$ contains $C_3$ and $K_{1,3}$; also $|V(\Gamma_c(H))|\geq 5$, then
$\Gamma_c(G)$ also has the same properties. So by Propositions~\ref{ct2} and \ref{ct5}, we need to check when $H\cong\mathbb Z_{p^4}$. If $H\cong \mathbb Z_{p^4}$,
then by~\eqref{e1}, $\Gamma_c(H)\cong K_3$, so $H$ together with its subgroups forms $K_4$ as a subgraph of $\Gamma_c(G)$. Also $|V(\Gamma_c(G))|\geq 5$,
since $G$ has a subgroup of order $q$ in addition.
\end{proof}

\begin{pro}\label{ct9}
 If $G$ is a non-abelian group of order $p^2q^2$, where $p,q$ are two distinct primes, then
$\Gamma_c(G)$ contains $C_3$ and $K_{1,3}$ as proper subgraphs; it has at least five vertices.
\end{pro}
\begin{proof}
 We use the classification of groups of order $p^2q^2$ given in \cite{lin}.
Let $P$ and $Q$ denote a Sylow $p,q$-subgroups of $G$ respectively. Without loss of generality, we assume that $p>q$.
By Sylow's Theorem, $n_p=1, q, q^2$. But $n_p= q$ is not possible,
since $p> q$. If $n_p= q^2$, then $p~|~(q+1)(q-1)$, this implies that $p~|~(q+1)$, which is true only
 when $(p, q)= (3, 2)$.

\noindent When $(p, q) \neq (3, 2)$, then $G \cong P \rtimes Q$. Now we have the following possibilities.

If $G \cong \mathbb Z_{p^2}\rtimes\mathbb Z_{q^2}=
\langle a,b~|~a^{p^2}=b^{q^2}=1,bab^{-1}=a^i,i^{q^2}\equiv 1~(\text{mod}~p^2)\rangle$, then $H_1:=\langle a\rangle$,
 $H_2:=\langle a^p\rangle$, $H_3:=\langle b\rangle$, $H_4:=\langle b^q\rangle$, $H_5:=\langle ab\rangle$ are proper cyclic subgroups of $G$,
so $|V(\Gamma_c(G))|\geq 5$.
Here $H_1$, $H_2$ are normal in $G$; $H_3$, $H_4$ permutes with each other. So $K_4$ is a proper subgraph of $\Gamma_c(G)$
induced by $H_1$, $H_2$, $H_3$, $H_4$.

If $G \cong \mathbb Z_{p^2} \rtimes (\mathbb Z_q \times \mathbb Z_q)$, then $H_1:=\langle a\rangle$,
 $H_2:=\langle a^p\rangle$, $H_3:=\langle b\rangle$, $H_4:=\langle c\rangle$, $H_5:=\langle bc\rangle$ are proper cyclic subgroups of $G$,
so $|V(\Gamma_c(G))|\geq 5$.
Here $H_1$ is a normal subgroup of $G$; $H_3$, $H_4$, $H_5$ permutes with each other. So $K_4$ is a proper
subgraph of $\Gamma_c(G)$ induced by $H_i$, $i=1,3,4,5$.

If $G \cong (\mathbb Z_p \times \mathbb Z_p) \rtimes \mathbb Z_{q^2}$ or $(\mathbb Z_p \times Z_p) \rtimes (\mathbb Z_q \times \mathbb Z_q )$,
 then $\mathbb Z_p \times \mathbb Z_p$ is a subgroup of $G$. Since $p>2$, so by \eqref{e1}, $\Gamma_c(G)$ contains $K_4$ as a proper subgraph and
so $|V(\Gamma_c(G))|\geq 5$.

Next, we consider the case when $(p, q)=(3,2)$ and $n_p= 1$. Consider the  Sylow 3-subgroup $P$ and a Sylow 2-subgroup $Q$ of $G$.
Let $H$ be a subgroup of $Q$ of order $2$. Since $|G|$ does not divide $[G: P]!$, so $P$ contains a subgroup, say
 $K$ of order $3$, which is normal in $G$; $H_1:=QK$ is a subgroup of order $12$.
Suppose $\Gamma_c(H)$ contains $C_3$ and $K_{1,3}$; also $|V(\Gamma_c(H))|\geq 5$, then
$\Gamma_c(G)$ also has the same. So by Propositions~\ref{ct2} and \ref{ct7}, the only cases remains to check is when $H\cong\mathbb Z_{p^2q}$ or $A_4$.
If $H_1\cong \mathbb Z_{p^2q}$, then by~\eqref{e1}, $\Gamma_c(H_1)\cong K_4$, so $H$ together with its subgroups forms $K_5$ as a proper subgraph of
$\Gamma_c(G)$ and so $|V(\Gamma_c(G))|\geq 5$. If $H_1\cong A_4$, then by~\eqref{e4}, $\Gamma_c(H_1)\cong K_3\cup \overline{K}_4$, so $|V(\Gamma_c(G))|\geq 5$. Also $K_{1,3}$ is a subgraph of $\Gamma_c(G)$ with bipartition $X:=\{K\}$ and $Y:=\{K_1, K_2, K_3\}$, where $K_i$'s are the vertices of $K_3$ in $\Gamma_c(H_1)$.
\end{proof}

\begin{pro}\label{ct10}
 If $G$ is a non-abelian group of order $p^{\alpha}q^{\beta}$, where $p,q$ are distinct primes, and $\alpha$, $\beta\geq2$, then
$\Gamma_c(G)$ has $C_3$ and $K_{1,3}$ as proper subgraphs; it has at least five vertices.
\end{pro}
\begin{proof}
 We prove the result by induction on $\alpha + \beta$. If $\alpha + \beta = 4$, then by Propositions~\ref{ct2} and~\ref{ct9}, the result is true in the case.
Assume that the result is true for all non-abelian groups  of order $p^mq^n$ with $m,n \geq 2$, and $m+n < \alpha + \beta$.
 We prove the result when $\alpha + \beta > 4$.  Since $G$ is solvable, $G$ has a subgroup $H$ of prime index, with out loss of generality, say $q$.
So $|H|= p^{\alpha}q^{\beta - 1}$. If $H$ is abelian, then by Proposition~\ref{ct2}, the result is true. If $H$ is non-abelian,
 then we have the following cases to consider:

\noindent \textbf{Case 1:} If $\beta = 2$, then $\alpha > 2$. So by Proposition~\ref{ct8}, the result is true for $\Gamma_c(H)$.

\noindent \textbf{Case 2:} If $\beta > 2$, then by induction hypothesis, the result is true for $\Gamma_c(H)$.

\noindent \textbf{Case 3:}  If $\alpha =2$, then $\beta > 2$. So by Case 2, the result is true for $\Gamma_c(H)$.

\noindent \textbf{Case 4:} If $\alpha >2$, then by induction hypothesis, the result is true for $\Gamma_c(H)$.

Then by Theorem~\ref{ct1}, result is true for $\Gamma_c(G)$ also.
\end{proof}

\begin{pro}\label{100}
Let $G$ be a finite group of order $p_1^{\alpha_1}p_2^{\alpha_2}\ldots p_k^{\alpha_k}$, $k\geq 3$, where $p_i$'s are distinct primes and $\alpha_i\geq 1$.
Then $\Gamma_c(G)$ contains $C_3$, $K_{1,3}$ and it has more than four vertices.
\end{pro}
\begin{proof}
If $\alpha_i=1$, for every $i$, then $G$ is solvable. We consider the following cases:

\noindent \textbf{Case 1:} $k =3$.
If $\alpha_1=\alpha_2=\alpha_3=1$, then
without loss of generality, we assume that $p_1<p_2<p_3$. Since $G$ is solvable, it has a Sylow basis $\{P_1, P_2, P_3\}$,
where $P_i$ is the Sylow $p_i$-subgroup of $G$ for every $i=1,2,3$. Also $H_1:=\langle ab\rangle$ and $H_2:=\langle bc\rangle$ are proper cyclic
subgroups of $G$, where $a$, $b$, $c$ are generators of  $P_1$, $P_2$, $P_3$ respectively, so we have $|V(\Gamma_c(G))|\geq 5$.
Moreover, $P_1$, $P_2$, $P_3$ permutes with each other, so $\Gamma_c(G)$ contains $C_3$ as a proper subgraph. Further,
$G$ has a normal subgroup, say $N$ of order $p_3$, so
it follows that $\Gamma_c(G)$ contains $K_{1,3}$ as a subgraph with bipartition $X:=\{N\}$ and $Y:=\{P_1,P_2,H_1\}$.

 \noindent \textbf{Case 2:} $k >3$. Since $G$ is solvable, it has a Sylow basis containing $P_1$, $P_2$, $P_3$ , where $P_i$ is the
Sylow $p_i$-subgroup of $G$ for every $i=1,2,3$. Then $H:=P_1P_2P_3$ is a subgroup of $G$. So by
Proposition~\ref{ct2} and by Case 1 of this proof, $\Gamma_c(H)$ contains $C_3$ and $K_{1,3}$ as  subgraphs; also $|V(\Gamma_c(H))|\geq 5$.
It follows that  $\Gamma_c(G)$ also has the same properties.

If $\alpha_i>1$, for some $i$, then without loss of generality, we assume that $\alpha_1>1$. By Sylow's theorem, $G$ has a Sylow $p_1$-subgroup, say $P$ and
$G$ has an element, say $b$ of order $p_2$. If $P$ is non-abelian, then by Proposition~\ref{ct5}, $\Gamma_c(P)$ contains $C_3$, $K_{1,3}$ as a subgraph. By
Theorem~\ref{ct5}, taking the cyclic subgroups of $P$ together with $\langle b\rangle$, we have $|V(\Gamma_c(G))|\geq 5$.

If $P$ is abelian, then we consider the following cases:

\noindent\textbf{Case 3:} $P$ is cyclic. Let $P:=\langle a\rangle$. Now consider the subgroup $\langle a, b\rangle$ of $G$. Then by Propositions~\ref{ct2}, \ref{ct6},
\ref{ct7}, \ref{ct8}, \ref{ct9}, and \ref{ct10}, we have $\Gamma_c(\langle a, b\rangle)$ contains $C_3$, $K_{1,3}$. Also by Propositions~\ref{ct6},
\ref{ct7}, \ref{ct8}, \ref{ct9} and \ref{ct10}, taking cyclic subgroups of $\langle a,b\rangle$ together with $\langle a, b\rangle$, we have $|V(\Gamma_c(G))|\geq 5$.

\noindent\textbf{Case 4:} $P$ is non-cyclic. If $\alpha_1=2$, then $P\cong \mathbb Z_p\times \mathbb Z_p:=\langle a_1, a_2\rangle$.

\noindent\textbf{subcase 4a:} If $\langle a_1, a_2, b\rangle\ncong A_4$, then by Propositions~\ref{ct2}, \ref{ct6},
\ref{ct7}, \ref{ct8}, \ref{ct9} and \ref{ct10}, $\Gamma_c(\langle a_1, a_2, b\rangle)$ contains $C_3$, $K_{1,3}$ and $|V(\langle a_1, a_2, b\rangle)|\geq 5$.

\noindent\textbf{subcase 4b:} If $\langle a_1, a_2, b\rangle \cong A_4$. By \eqref{e4}, $\Gamma_c(\langle a_1, a_2, b\rangle)$ has $C_3$ as a subgraph.
Let $c$ be an element of $G$ of order $p_3$. If $\langle c\rangle$ permutes with
a cyclic subgroups of $\langle a_1, a_2\rangle$, then $\Gamma_c(\langle a_1, a_2\rangle)\cong C_3$. So $\langle c\rangle$ together with cyclic subgroups
of $\langle a_1, a_2\rangle$ forms $K_{1,3}$. If $\langle c\rangle$ doesnot permutes with a subgroups of $\langle a_1, a_2\rangle$, the by
Propositions~\ref{ct2}, \ref{ct6},
\ref{ct7}, \ref{ct8}, \ref{ct9} and \ref{ct10}, $\Gamma_c(\langle a_1, c\rangle)$ contains $K_{1,3}$ as a subgraph. Also $|V(\langle a_1, a_2, b\rangle)|\geq 5$, since
by \eqref{e4}. If $\alpha_1\geq 3$, then by Proposition~\ref{ct2}, the result is true for $\Gamma_c(P)$, so it is true for $\Gamma_c(G)$ also.

The proof follows by combining all these cases.
\end{proof}

\subsection{Main results for finite groups}\label{sec:7}
Combining all the results obtained so-far in this section, we have the following main results which are applications for group theory.

\begin{thm}\label{ct14}
 Let $G$ be a finite group and $p$, $q$ be distinct primes. Then
\begin{enumerate}[{\normalfont (i)}]
\item $\Gamma_c(G)$ is $C_3$-free if and only if $G$ is  one of $\mathbb Z_{p^{\alpha}}$ $(\alpha = 2,3)$, $\mathbb Z_{pq}$, $\mathbb Z_q\rtimes \mathbb Z_p$;
\item $\Gamma_c(G)$ is $C_n$ if and only if $n=3$ and $G$ is  either $\mathbb Z_{p^4}$ or $\mathbb Z_2\times \mathbb Z_2$;
\item $\Gamma_c(G)$ is $P_n$ if and only if $n=1$ and $G$ is  either $\mathbb Z_{p^3}$ or $\mathbb Z_{pq}$;
\item $\Gamma_c(G)$ is $K_4$ if and only if $G$ is  one of $\mathbb Z_{p^5}$, $\mathbb Z_{p^2q}$, $\mathbb Z_3\times \mathbb Z_3$, $Q_8$;
\item $\Gamma_c(G)$ is claw-free if and only if $G$ is  one of $\mathbb Z_{p^\alpha}$ $(\alpha=2,3,4)$, $\mathbb Z_{pq}$, $\mathbb Z_2 \times \mathbb Z_2$,
 $A_4$.
\end{enumerate}
\end{thm}

\begin{cor}\label{pc15}
 Let $G$ be a finite group and $p$, $q$ are distinct primes.
\begin{enumerate}[{\normalfont (i)}]
\item The following are equivalent:
\begin{enumerate}[{\normalfont (a)}]
\item $\Gamma_c(G)$ is $C_3$-free;
\item $\Gamma_c(G)$ is bipartite;
\item $\Gamma_c(G)$ is complete bipartite;
\item $\Gamma_c(G)$ is tree;
\item $\Gamma_c(G)$ is star graph.
\end{enumerate}
\item $\Gamma_c(G)$ is $P_2$-free if and only if $G$ is  either $\mathbb Z_{p^{\alpha}}$ $(\alpha = 2,3)$ or $\mathbb Z_{pq}$.
\item $girth(\Gamma_c(G))$ is infinity if $G$ is  one of $\mathbb Z_{p^{\alpha}}$ $(\alpha = 2,3)$, $\mathbb Z_{pq}$ or $\mathbb Z_q\rtimes \mathbb Z_p$;
otherwise $girth(\Gamma_c(G))=3$.
\end{enumerate}
\end{cor}
\begin{proof}
To classify the groups whose permutability graph is either bipartite or complete bipartite, it is enough to consider the groups whose
permutability graph of cyclic subgroups are $C_3$-free. By Theorem~\ref{ct14}(i) and \eqref{e5}, \eqref{e3}, we have $(a)\Leftrightarrow(b)
\Leftrightarrow (c)$.
Now, to classify the groups whose permutability graphs of cyclic subgroups is one of
tree, star graph or $P_2$-free, it is enough to
consider the groups whose permutability graphs of cyclic subgroups are bipartite. So by the above argument and
by Theorem~\ref{ct14}(i), \eqref{e5}, \eqref{e3}, we have $(b)\Leftrightarrow(d)\Leftrightarrow(e)$ and $\Gamma_c(G)$ is $P_2$-free if and only if
$\mathbb Z_{p^\alpha}(\alpha=2,3)$ or $\mathbb Z_{pq}$. This completes the proof of (i) and (ii).
The proof of (iii) follows by the part (i) of this corollary and by Theorem~\ref{ct14}(i).
%
%
%
\end{proof}

\begin{cor}\label{ct102}
Let $G$ be a finite group. Then $\Gamma_c(G)$ is totally disconnected if and only if $G\cong \mathbb Z_{p^2}$.
\end{cor}
\begin{proof}
Let $|G|=p_1^{\alpha_1}p_2^{\alpha_2}\ldots p_k^{\alpha_k}$, where $p_i$'s are distinct primes, $k\geq 1$ and $\alpha_i\geq 1$. If
$\alpha_i=1$, for every $i$, then $G$ is solvable. Suppose $k=1$, then $G$ does not contains a proper subgroup. It follows that
$k\geq 2$ and so any two subgroups in Sylow basis of $G$ permutes with each other. Therefore, $\Gamma_c(G)$ is not totally disconnected.
If $\alpha_i>1$, for some $i$, then without loss of generality we
assume that $\alpha_1>1$ and so by Sylow's Theorem, $G$ has a Sylow $p_1$ subgroup, say $P$. Suppose $P\ncong \mathbb Z_{p^2}$, then
by Propositions~\ref{ct2} and \ref{ct5}, $\Gamma_c(G)$ is not totally disconnected. If $P\cong \mathbb Z_{p^2}$, then $P$ and its subgroup of order $p$
permutes with each other. Thus $\Gamma_c(G)$ is not totally disconnected.
%
%
\end{proof}

\begin{rem}Not every graph is a permutability graph of cyclic subgroups of some group. For example, by Theorem~\ref{ct14} (3), the graph $C_n$, $n \geq 4$ is not a permutability graph of cyclic subgroups of any group.
\end{rem}

\subsection{Infinite groups}\label{sec:8}
We now investigate the  of permutability graph of cyclic subgroups of infinite groups.
 It is well known that any infinite
group has infinite number of subgroups. Let $G$ be an infinite abelian group.
If $G$ is finitely generated, then by fundamental theorem of finitely generated abelian groups, $\mathbb Z$
is a subgroup of $G$. Since $\mathbb Z$ is cyclic, it follows that $\Gamma_c(\mathbb Z)$ contains $K_r$
as a proper subgraph for every positive integer $r$. Therefore, by Theorem~\ref{ct1}, $\Gamma_c(G)$ also has the same property.
If $G$ is not finitely generated, then we can take the cyclic groups generated by each generating element and so  $\Gamma_c(G)$ contains $K_r$ as a proper subgraph, for every positive integer $r$. Thus we have the following result.

\begin{thm}\label{ct17}
The permutability graph of cyclic subgroups of any infinite abelian group contains $K_r$ as a subgraph, for every positive integer $r$.
\end{thm}

Next, we consider the infinite non-abelian groups.
Recall that an infinite non-abelian group $G$  in which every proper subgroups  of $G$
have order a fixed prime number $p$ is called a \textit{Tarski monster group}. Existence of such groups was given by  Ol'shanskii in \cite{olshan}.
In general, the existence of infinite non-abelian groups in which the order of all proper subgroups are of prime order (primes not necessarily distinct)
were also given by him in \cite[Theorem 35.1]{olshan1}. Also M. Shahryari in \cite[Theorem 5.2]{shah} give existence of countable non-abelian simple groups with the property that
their all non-trivial finite subgroups are cyclic of order a fixed prime $p$ (of course this existence can also be deduced from the results of \cite{olshan1}). It is easy to see that the permutability graph of cyclic subgroups of the above mentioned first two class of non-abelian groups are totally disconnected and for the third class of non-abelian groups it is totally disconnected if that group does not have $\mathbb Z$ as a subgroup. In the next result,
we characterize the infinite non-abelian groups whose permutability graph of cyclic subgroups is totally disconnected.

\begin{thm}\label{ct108}
Let $G$ be an infinite group. Then $\Gamma_c(G)$ is totally disconnected if and only if every non-trivial finite subgroup of $G$ is of prime order (primes not necessarily distinct) and $\mathbb Z$ is not a subgroup of $G$.
\end{thm}
\begin{proof}
It is easy to see that if every proper subgroup of $G$ is of prime order (primes not necessarily distinct) and $\mathbb Z$ is not a subgroup of $G$ then $\Gamma_c(G)$ is totally disconnected.
Conversely, suppose that $\Gamma_c(G)$ is totally disconnected. Then by Theorem~\ref{ct17}, $G$ must be non-abelian.
Suppose not every proper subgroup of $G$ is of prime order, then we have the following possibilities.

(i) $G$ may have a subgroup whose order is a composite number; or

(ii) all the subgroups of $G$ may have infinite order.

If $G$ is of type (i), then let $H$ be a subgroup of $G$ of composite order. If $H \ncong Z_{p^2}$, then by Corollary~\ref{ct102},
$\Gamma_c(H)$ is not totally disconnected.
If $H \cong Z_{p^2}$, then $H$ and its subgroup of order $p$ permutes with each other. So it follows that $\Gamma_c(G)$ is not totally disconnected.

If $G$ is of type (ii), then it must have $\mathbb Z$ as a subgroup  and so by Theorem~\ref{ct17},  $\Gamma_c(G)$ is not totally disconnected.
Hence the proof.
\end{proof}
%

\section{Further results on $\Gamma_c(G)$}\label{sec:9}
Recall that a subgroup $H$ of a group $G$ is said to be \emph{permutable }if it permutes with all the subgroups of $G$.
In \cite{iwa}, Iwasawa characterized the groups whose subgroups are permutable.
\begin{thm}\label{ct30}\textbf{(\cite{iwa})}
 A group whose subgroups are permutable is a nilpotent group
in which for every Sylow $p$-subgroup $P$, either $P$ is a direct product of a
quaternion group and an elementary abelian 2-group, or $P$ contains an
abelian normal subgroup $A$ and an element $b\in P$ such that $P=A\langle b\rangle$ and
there exists a natural number $s$, with $s\geq2$ if $p= 2$, such that $a^b= a^{1+p^s}$
for every $a\in A$.
\end{thm}

The next theorem classifies the groups whose permutability graph of cyclic subgroups are complete (see, also in \cite[p.14]{bal}).
\begin{thm}\label{ct50}
 Let $G$ be a group. Then $\Gamma_c(G)$ is complete if and only if $G$ is one of the groups given in Theorem~\ref{ct30}.
\end{thm}

\begin{thm}\label{ct19}
 Let $G$ be a group with a permutable proper cyclic subgroup. Then $\Gamma_c(G)$ is regular if and only if $\Gamma_c(G)$ is complete.
\end{thm}
\begin{proof}
Let $N$ be a permutable cyclic subgroup of $G$.  Assume that $\Gamma_c(G)$ is regular.  Since $N$ permutes with all the cyclic subgroup of $G$,
so from the regularity of $\Gamma_c(G)$, it follows that
any two vertices in $\Gamma_c(G)$ are adjacent and hence $\Gamma_c(G)$ is complete. Converse of the result is obvious.
\end{proof}

\begin{thm}\label{ct18}
 Let $G$ be a group with a permutable proper cyclic subgroup. Then $\Gamma_c(G)$ is connected and  $diam(\Gamma_c(G))\leq2$.
\end{thm}
\begin{proof}
If every cyclic subgroups of $G$ are permutable, then obviously $\Gamma_c(G)$ is connected and $diam (\Gamma_c(G))=1$. Let $N$ be a permutable proper cyclic
subgroup of $G$. Suppose $H$ and $K$ are two proper cyclic subgroups of $G$ such that $HK\neq KH$. Then we have a path $H-N-K$ in $\Gamma_c(G)$ and so
 $\Gamma_c(G)$ is connected and $diam(\Gamma_c(G))=2$.
\end{proof}

\begin{prob} Which groups have connected permutability graph of cyclic subgroups $?$ and estimate their diameter.
\end{prob}
In the next result, we classify  the abelian groups  whose permutability graph of cyclic subgroups are planar.

\begin{thm}\label{ct16}
 Let $G$ be an abelian group and $p$, $q$ be distinct primes. Then $\Gamma_c(G)$ is planar if and only if $G$ is isomorphic to one of the following:
 $\mathbb Z_{p^\alpha} (\alpha=2,3,4,5)$,
 $\mathbb Z_{pq}, \mathbb Z_{p^2 q}$,
 $\mathbb Z_2\times \mathbb Z_2$, $\mathbb Z_3\times \mathbb Z_3$.

\end{thm}
\begin{proof}If $G$ is infinite abelian, then by Theorem \ref{ct17}, $\Gamma_c(G)$ is non-planar. So in the rest of the proof, we assume that $G$ is finite.

 Suppose $G$ is cyclic, then with the notations used in the proof of Proposition  \ref{ct2} and by \eqref{e5}, we have $\Gamma_c(G)\cong K_r$. So $\Gamma_c(G)$ is planar if and only if $r\leq 4$.
This is true only when one of the following holds:
\begin{enumerate}
\item[(i)] $k=1$ with $\alpha_1<6$;
\item[(ii)] $k=2$ with $\alpha_1=1,\alpha_2=1$;
\item[(iii)] $k=2$ with $\alpha_1=2,\alpha_2=1$.
\end{enumerate}
If $G$ is non-cyclic, then we need to consider the following cases:

\noindent\textbf{Case 1}: $G \cong \mathbb Z_p \times \mathbb Z_p$. Then the number of proper  subgroups of $G$ is $p+1$;
 they are $\langle (1,0) \rangle$, and $\langle x, 1\rangle$, $x \in \{0,1,\ldots ,p-1  \}$. By \eqref{e1}, $\Gamma_c(G)$ is
planar only when $p=2,3$.

\noindent\textbf{Case 2:} $G\cong{\mathbb Z_{p^2}} \times{\mathbb Z_p}$. Then $\langle(1, 0)\rangle$, $\langle (1,1) \rangle$, $\langle (p,0) \rangle$,
$\langle (0,1) \rangle$, $\langle (p,1) \rangle$ are proper subgroups of $G$, so $\Gamma_c(G)$ contains $K_5$ as a subgraph.

\noindent\textbf{Case 3:} $G\cong \mathbb Z_{pq}\times \mathbb Z_p$. Then $\mathbb Z_{pq}$, $\mathbb Z_q$, $\mathbb Z_p\times \mathbb Z_p$ are proper subgroups of $G$. Here
$\mathbb Z_p\times \mathbb Z_p$ has at least three proper subgroups of order $p$, so these three subgroups together with $\mathbb Z_{pq}$,
$\mathbb Z_q$ forms $K_5$ as a subgraph of $\Gamma_c(G)$.

\noindent\textbf{Case 4:} $G \cong \mathbb Z_{p^k} \times \mathbb Z_{p^l}$, $k$, $l\geq 2$.
Then $\mathbb Z_{p^2}\times \mathbb Z_p$ is a proper subgroup of $G$,
so by Case 2 and by Theorem~\ref{ct1}, $\Gamma_c(G)$ contains $K_5$ as a subgraph.

\noindent\textbf{Case 5:} $G \cong\mathbb Z_p\times\mathbb Z_p\times\mathbb Z_p$. then $G$ has two subgroups each isomorphic to
$\mathbb Z_p\times \mathbb Z_p$. It follows that $G$ has at least five subgroups of order $p$ and so they form $K_5$ as a subgraph of $\Gamma_c(G)$.

\noindent\textbf{Case 6:} $G \cong {\mathbb Z}_{p_1^{\alpha_1}} \times {\mathbb Z}_{p_2^{\alpha_2}} \times \ldots \times {\mathbb Z_{p_k^{\alpha_k}}}$, where
$p_i$'s are primes and $\alpha_i\geq1$. If $k=2$ or 3, then $\alpha_i>1$, for some $i$ and if $k\geq 4$, then $\alpha_i\geq 1$. In either case, one of
$\mathbb Z_{p^2}\times \mathbb Z_p$, $\mathbb Z_{pq}\times \mathbb Z_p$, $\mathbb Z_p \times \mathbb Z_p \times \mathbb Z_p$ is a proper subgroup of $G$,
so by Cases 2, 3 and 5, $\Gamma_c(G)$ contains $K_5$ as a subgraph.

The result follows by combining all the above cases.
\end{proof}
Proposition~\ref{ct6} shows the existence of a finite non-abelian group whose permutability graph of cyclic subgroups is planar. Further, the Torski monster group is an example of an infinite non-abelian group whose permutability graph of cyclic subgroups is planar. Now we pose the following

\begin{prob} Classify all non-abelian groups whose permutability graph of cyclic subgroups are planar.
\end{prob}

The next result characterize some non-abelian groups by using their permutability graph of cyclic subgroups.
\begin{thm}\label{ct51}
Let $G$ be a finite group.
\begin{enumerate}[{\normalfont (i)}]
\item If $G$ is non-abelian and $\Gamma_c(G)\cong \Gamma_c(Q_8)$, then $G\cong Q_8$.
\item If $\Gamma_c(G)\cong \Gamma_c(S_3)$, then $G\cong S_3$.
\item If $\Gamma_c(G)\cong \Gamma_c(A_4)$, then $G\cong A_4$.
\end{enumerate}
\end{thm}
\begin{proof}
\begin{itemize}
\item [(i);] By Theorem~\ref{ct14}(5), $Q_8$ is the only non-abelian group such that $\Gamma_c(Q_8)=K_4$, so the result follows.
\item  [(ii):] By Theorem~\ref{ct14}(1) and \eqref{e5}, \eqref{e3}, $S_3$ is the only group such that $\Gamma_c(S_3)=K_{1,3}$, so the result follows.
\item [(iii):] By Theorem~\ref{ct14}(6) and \eqref{e5}, \eqref{e1}, \eqref{e4}, $A_4$ is the only group such that $\Gamma_c(A_4)=K_3\cup \overline{K}_4$, so the result follows.
\end{itemize}
\end{proof}

\section*{Acknowledgements}
The authors would like to thank Professor A. Yu. Ol'shanskii,  Department of Mathematics,
Vanderbilt University, Nashville, Tennessee, USA for pointing out the reference \cite{olshan1} to our attention.

\end{document}